\documentclass[11pt,a4paper]{article}

\usepackage{amsmath,amsfonts,amssymb,amsthm}
\allowdisplaybreaks
\usepackage{graphicx}
\usepackage{enumitem}

\usepackage[margin=2.5cm]{geometry} 
\usepackage[T1]{fontenc}
\usepackage{authblk}

\newcommand{\keywords}[1]{\par\addvspace\baselineskip\noindent\textbf{Keywords:}\enspace#1}
\newcommand{\msc}[1]{\par\addvspace{0.5\baselineskip}\noindent\textbf{MSC Classification:}\enspace#1}

\usepackage{hyperref}

\usepackage{xcolor}

\def\intO{\int_\Omega}
\def\tmax{T_{\text{max}}}
\def\R{\mathbb{R}}
\def\N{\mathbb{N}}

\usepackage{constants}

\newconstantfamily{const}{
	symbol=c, 
}
\newconstantfamily{eps}{
	symbol=\varepsilon, 
}
\newconstantfamily{tau}{
	symbol=\tau, reset={section}
}

\newtheorem{theorem}{Theorem}[section]
\newtheorem{lemma}[theorem]{Lemma}
\newtheorem{remark}[theorem]{Remark}

\begin{document}
	
\title
{Nonlocal logistics and nonlinear productions in an attraction-repulsion chemotaxis model: \\ analysis of the global well-posedness}

\author[1]{{Rafael} {D\'iaz Fuentes}\thanks{rafael.diazfuentes@unica.it}}

\author[2]{{Mar\'ia Victoria} {Redondo Neble}\thanks{victoria.redondo@gm.uca.es}}

\author[1]{{Giuseppe} {Viglialoro}\thanks{giuseppe.viglialoro@unica.it}}

\affil[1]{{Dipartimento di Matematica e Informatica}, {Universit\`a degli Studi di Cagliari}, \protect\\ {{Via Ospedale 72}, {Cagliari}, {09124}, {(CA)}, {Italy}}}

\affil[2]{{Departamento de Matemáticas}, {Universidad de Cádiz, C.A.S.E.M.}, \protect\\ {{Avda. Rep\'ublica \'Arabe Saharawi}, {Puerto Real}, {11510}, {Cádiz}, {Spain}}}

\date{\today}
	\maketitle
	\begin{abstract}
    This paper investigates a {{three-component}} chemotaxis system involving both attraction and repulsion effects, as well as a nonlocal logistic-type source term. Mathematically, if $u=u(x,t)$, $v = v(x,t)$ and $w = w(x,t)$ denote the cell distribution, and   
the attractive and the repulsive chemical signals, the model is then described by
	\begin{equation*}
		\begin{cases} 
			u_t = \Delta u - \chi \nabla \cdot (u \nabla v) + \xi \nabla \cdot (u \nabla w) + a u^\alpha - b u^\alpha \int_\Omega u^\beta, & x \in \Omega, \ t > 0, \\ 
			\tau v_t = \Delta v - v + f(u), & x \in \Omega, \ t > 0, \\ 
			\tau w_t = \Delta w - w + g(u), & x \in \Omega, \ t > 0. 
		\end{cases}
	\end{equation*}
    Here, $\Omega \subset \mathbb{R}^n$ ($n \geq 1$) is a bounded smooth domain, $\tau\in\{0,1\}$, $a,b,\alpha,\beta,\chi,\xi>0$, the production functions $f(u)$ and $g(u)$ are assumed to satisfy algebraic growth conditions of order $\ell$ and $\rho$, generalizing prototypes of the form $u^\ell$ and $u^\rho$, $\ell,\rho>0$. The work is devoted to proving the global existence and boundedness of classical solutions under a suitable balance between the signal production exponents $\ell, \rho$ and the nonlocal damping exponents $\alpha, \beta$, for regular enough initial data and zero-flux boundary restrictions. In this regard, two main theorems are established for the cases where the chemical signals satisfy either elliptic ($\tau=0$) or parabolic ($\tau=1$) partial differential equations, highlighting how sufficiently strong nonlocal damping prevents the formation of singularities in time.
    We extend the results obtained in \cite{CDFV24}, {{where the fully parabolic ($\tau=1$) and only attraction version is studied. In our context, we establish }} well-posedness of the system and the long-time behavior of solutions. 
	\end{abstract}
	
	
	\msc{35A01, 35R09, 35B45, 35K55, 35Q92; Secondary: 92C17.}

 	\tableofcontents
	
	\section{Introduction and motivations}

    \subsection{Some comments on chemotaxis phenomena}
	Chemotaxis, the directed movement of cells or organisms in response to chemical gradients, is a ubiquitous phenomenon in biology, underlying processes from bacterial patterning and immune response to embryogenesis and tumor invasion. The complex interplay between cell migration, proliferation, and chemical signaling dynamics gives rise to intricate spatio-temporal structures and behaviors. Mathematical models, particularly systems of partial differential equations, are indispensable tools for elucidating the mechanisms governing these phenomena and for predicting their emergent properties.
	
	The foundational Keller-Segel (KS) models \cite{Keller1970,Keller1971} provided an early framework for describing chemotaxis driven by attractive signals. However, the classical KS systems are well-known to exhibit solutions that form finite-time singularities, which often limits its applicability in biological contexts where cell populations are observed to remain bounded. 

By the mathematical point-of-view a general formulation to describe chemotaxis phenomena is given by 
\begin{equation}\label{eq:KS_general}
\begin{cases}
    u_t = \nabla \cdot \big( D(u)\nabla u - \chi(u,v)\nabla v \big) + f(u,v), 
    & \text{in}\; \Omega\times(0,\tmax), \\
    v_t = \nabla \cdot \big( D_v(v)\nabla v \big) + g(u,v), 
    & \text{in}\; \Omega\times(0,\tmax), \\
    u(x,0)=u_0(x), \quad v(x,0)=v_0(x), 
    & x \in \Omega,
\end{cases}
\end{equation}
where $\tmax$ denotes the maximum time up to which solutions to the system are defined.

In the above model, the spatial domain is denoted by $\Omega \subset \mathbb{R}^n$. 
The function $u = u(x,t)$ represents the cell density (for instance, bacteria or biological cells), whereas $v = v(x,t)$ denotes the concentration of the chemoattractant, which transports the cells in the direction of its gradient.  
The term $D(u)$ corresponds to the diffusion coefficient of the cells, which may depend nonlinearly on the density $u$. 
The function $\chi(u,v)$ describes the chemotactic sensitivity. 
Moreover, $D_v(v)$ denotes the diffusion coefficient of the chemoattractant. 
The term $f(u,v)$ represents a source term in the equation for $u$, possibly incorporating effects such as growth, death, competition, or external forcing. 
In turn, $g(u,v)$ accounts for the production and/or degradation mechanisms of the chemoattractant. 
Finally, $u_0(x)$ and $v_0(x)$ denote the prescribed initial data.

The qualitative behavior of system \eqref{eq:KS_general} is highly sensitive to the choice of initial data and to the specific structure of the functions $D$, $D_v$, $f$, $g$, and $\chi$; In particular, depending on the specific choices of parameters and functional forms, solutions may exhibit finite-time blow-up (gathering mechanisms for the cell density), global-in-time boundedness and stability, or the development of intricate nontrivial patterns. For a detailed analysis of models closely related to \eqref{eq:KS_general}, we refer, for instance, to \cite{Horstmann2005, Bellomo2015, Nagai01, Nagai1995blow, HerreroVelazquez, Tello2007, Winkler2010,FuestLankeitMaizu2025}.

\subsection{Some motivations: the introduction of nonlocal sources}

In this article, we focus on regimes in which aggregation phenomena are suppressed. 
One possible mechanism preventing blow-up consists in incorporating an additional chemical signal $w=w(x,t)$ acting as a chemorepellent, whose effect counterbalances the attraction induced by $v$. 
The resulting model naturally leads to a three-component system of partial differential equations that extends \eqref{eq:KS_general}. 

From a biological perspective, such an extension is well motivated, as cells in many systems respond not to a single chemical stimulus but to complex chemical environments involving both attractive and repulsive cues. 
For example, immune cells are attracted toward pathogens while simultaneously being repelled by inhibitory signals, and metastatic tumor cells may follow growth-promoting factors while avoiding cytotoxic regions. 
Mathematical models that incorporate both attraction and repulsion provide a more comprehensive framework for cellular migration and are known to exhibit a wider spectrum of dynamical behaviors, including stable pattern formation, travelling waves, and oscillatory dynamics that do not arise in purely attractive settings \cite{Luca2003,Tao2013,Li2016,Viglia21,FrassuCorViglialoro,YutaroTomomi22,YutaroTomomi26}.

Over the past decades, classical local logistic sources have been extensively studied in the context of biological transport and chemotaxis models, with results reported in investigations already mentioned above. In contrast, nonlocal logistic sources are attracting growing interest only recently, as they capture interactions based on spatially averaged quantities rather than pointwise values. These nonlocal terms arise in population dynamics and, through their nonlinearities, can either stabilize the system or produce complex spatio-temporal patterns.

From an analytical perspective, nonlocal and nonlinear source terms often play a key role in preventing blow-up and ensuring bounded solutions in chemotaxis systems. While recent studies (e.g., \cite{Bian18,Chai,CDFV24,Viglia21,ZHAO25}) show that such mechanisms can regularize dynamics and still allow nontrivial spatial patterns, here we focus specifically on attraction–repulsion systems with nonlocal and nonlinear source terms.

\subsection{Mathematical description of the model}
    
Building upon these foundational advancements, this paper is dedicated to the  mathematical analysis of a natural  three-component chemotaxis model featuring both attraction and repulsion, coupled with a nonlocal logistic-type source term. The system under consideration describes the spatio-temporal evolution of cell density $u(x,t)$, an attractive chemical signal $v(x,t)$, and a repulsive chemical signal $w(x,t)$ in an arbitrary bounded domain $\Omega \subset \R^n$, with $n\ge 1$ and with smooth boundary $\partial \Omega$. The model is given by
	\begin{equation}
		\begin{cases}
			u_t = \Delta u - \chi \nabla \cdot (u\nabla v) + \xi \nabla \cdot (u\nabla w) + a u^\alpha - b u^\alpha \intO u^\beta, & \text{in}\ \Omega\times(0,\tmax), \\
			\tau v_t = \Delta v - v + f(u), & \text{in}\  \Omega\times(0,\tmax), \\
			\tau w_t = \Delta w - w + g(u), & \text{in}\ \Omega\times(0,\tmax), \\
			\frac{\partial u}{\partial \nu} = \frac{\partial v}{\partial \nu} = \frac{\partial w}{\partial \nu} = 0, & \text{on}\ \partial \Omega \times(0,\tmax), \\
			u(x,0) = u_0(x), {{\tau}}v(x,0) = {{\tau}}v_0(x), {{\tau}}w(x,0) = {{\tau}}w_0(x), & \text{in}\ \bar{\Omega},
		\end{cases}		\label{eq:model}
	\end{equation}
with $\tau\in\{0,1\}$, {$\chi,\xi, a,b,\alpha,\beta > 0$}, and functions $f, g$, sufficiently regular,  obeying some growth conditions, behaving (for some $\ell,\rho>0$) as $u^\ell$ and $u^\rho$ respectively. The vector $\nu$ indicates the outward normal on $\partial \Omega$. The initial data $u_0(x)$, $v_0(x)$, and $ w_0(x)$ are assumed to be non-negative and sufficiently regular to ensure the existence of classical solutions.

Given the discussion so far, the model admits a natural interpretation. It fully accounts for the dynamic coupling of cells with both an attractant and a repellent, making it particularly relevant in scenarios where chemical gradients evolve on timescales comparable to cell movement, with the repellent actively shaping the environment. By incorporating nonlinear population kinetics ($\alpha$, $\beta$) and chemical production exponents ($\ell$, $\rho$), the model captures the interplay between aggregation, repulsion, and nonlocal damping, thereby providing a rigorous framework to study pattern formation, population persistence, and complex migratory behaviors.

The model herein proposed extends the existing literature in some aspects. While the structure of the $u$ and $v$ equations, including the nonlocal logistic term, shares similarities with \cite{CDFV24} ($\xi = 0$, $\ell=1$ in \eqref{eq:model}), and the attraction-repulsion aspect has precedents in models like \cite{Bian18}, our system uniquely couples these features within a three-component framework.

\subsection{Objectives and presentations of the main results}
The objective of this paper is to establish the fundamental well-posedness properties of solutions to system \eqref{eq:model}. We aim to prove the global existence and boundedness of classical solutions under some parameter constraints and initial conditions.  Our analysis, essentially, aims at  understanding how the relative strengths of attraction and repulsion (controlled by $\chi$ and $\xi$), combined with the parameters of the nonlocal logistic term, dictate the qualitative features of cell populations in environments with multiple dynamic chemical cues. 

{In order to present the results obtained in this investigation, let us fix our setting:
\begin{equation} \label{eq:modelCond}
	\begin{split}
		&a)\quad \Omega\subset\R^n, n \ge 1,\ \text{is a bounded domain of class } C^{2+\sigma},   \text{ for some} \ \sigma\in(0,1) \\
		&b)\quad \chi,\xi,a,b>0, \alpha,\beta\ge 1 \\
		&c)\quad  f,g\in C^1([0,\infty))\ 
        \\
		&d)\quad 0\le u_0, \tau v_0, \tau w_0 \in C^{2+\sigma}(\bar{\Omega}), \ 
        \text{and} \ 
		\partial_{\nu} u_0 = \partial_{\nu} v_0 = \partial_{\nu} w_0 = 0 \ \text{on} \ \partial \Omega.
	\end{split}
\end{equation}
Additionally, $f$ and $g$ also are contextually required to satisfy 
\begin{equation} \label{eq:fg}
	0 \leq f(s) \leq K_1 s^{\ell}, \qquad K_2 s^\rho \leq  g(s) \leq K_2 s(s+1)^{\rho-1} \qquad \forall s \ge 0,
\end{equation} 
where $K_1, K_2, \ell>0$ and $\rho > 1$.
}

The main results of this work are the following. 

\begin{theorem}[Parabolic-Elliptic model] \label{theo:PE}
	{{For $\tau=0$, let}}  {us require the conditions in \eqref{eq:modelCond} and \eqref{eq:fg}}. Then, whenever either
\begin{equation}  \label{eq:TheoPEcond}
	\begin{split}
		& a) \quad \beta > \frac12 n(\alpha-1) \quad \text{and} \quad \ell \le \min\{\alpha-1, \rho\}, \;{ or } \\ 
		& b) \quad \beta > \frac12 (n\ell + 2(\ell - \alpha+1)) \quad \text{and} \quad \alpha-1 < \min\{\ell, \rho\},
	\end{split}
\end{equation}
{are complied,}  it holds that problem \eqref{eq:model} admits a unique nonnegative solution
\begin{equation*}
	\begin{split}
		&
		u \in C^{2+\sigma,1+\frac\sigma2}(\bar{\Omega}\times[0,\infty)) \cap L^\infty({\Omega}\times(0,\infty)) \\		
		& 
		v,w \in C^{2+\sigma,\frac\sigma2}(\bar{\Omega}\times[0,\infty)) \cap L^\infty({\Omega}\times(0,\infty)).
	\end{split}
\end{equation*}
\end{theorem}

\begin{remark}
	The conditions in Theorem \ref{theo:PE} delineate two biological regimes where the system’s stabilizing mechanisms—chemorepulsion and nonlocal death—successfully prevent overcrowding (blow-up). Condition $a)$ addresses a \emph{growth-dominated} regime where attractant production $\ell$ is relatively weak and subordinated to both the proliferation rate $\alpha$ and the chemorepellent production $\rho$. In this case, global resource competition $\beta$, serves as the primary regulator. Conversely, {c}ondition $b)$ describes a \emph{signal-dominated} regime where attractant production is more aggressive. Here, the chemorepellent exponent $\rho$ must remain sufficiently high to provide a necessary spatial buffer against aggregation, while the nonlocal death rate $\beta$ must scale even more strictly to compensate for the heightened signaling intensity. In both regimes, the requirement $\min\{\ \cdot \ ,\rho\}$ highlights that chemorepulsion is an indispensable partner to nonlocal damping in maintaining the structural stability of the population.
\end{remark}

\begin{theorem}[Fully Parabolic model] \label{theo:PP}
	{For $\tau=1$, let} {us require the conditions in \eqref{eq:modelCond} and \eqref{eq:fg}.} Then, whenever either
	\begin{equation} \label{eq:TheoPPcond}
		\begin{split}
		& a) \quad \beta > \frac12 n(\alpha-1) \quad \text{and} \quad \alpha-1 \ge \max\{\rho,\ell\},\; {or }  \\
		& b) \quad \beta > \frac12 (n\ \max\{\rho,\ell\} + 2 (\max\{\rho,\ell\} - \alpha + 1)) \quad \text{and} \quad \alpha-1 < \min \{\rho,\ell\}
		\end{split}
	\end{equation}
{are complied,}	it holds that problem \eqref{eq:model} admits a unique nonnegative solution $(u,v,w)$ with
	\begin{equation*}
			u,v,w \in C^{2+\sigma,1+\frac{\sigma}{2}}(\bar{\Omega}\times[0,\infty)) \cap L^\infty(\Omega\times(0,\infty)).
	\end{equation*}
	
\end{theorem}

\begin{remark}
	In the fully parabolic regime ($\tau=1$) the chemical signals $v$ and $w$ possess their own evolutionary dynamics, introducing a finite propagation speed and a temporal lag in the cellular response. Mathematically and biologically, this is reflected in the conditions $a)$ and $b)$ in Theorem \ref{theo:PP} by the dominance of the term $\max\{\rho,\ell\}$. Condition $a)$ characterizes a \emph{proliferation-dominated} regime where both signaling mechanisms are relatively mild; here, the nonlocal damping $\beta$ is primarily tasked with regulating local birth rates. In contrast, condition $b)$ represents a \emph{high-intensity signaling} regime where at least one chemical signal (whether the aggregative attractant $\ell$ or the dispersive repellent $\rho$) is highly potent. In this scenario, the requirement on the nonlocal death rate $\beta$ becomes significantly more stringent to counteract the heightened cross-diffusive fluxes. This highlights that in a parabolic environment, the presence of strong signaling gradients necessitates a more robust global competitive pressure to maintain population stability, regardless of whether the primary signaling drive is attractive or repulsive.
\end{remark}

{
\begin{remark}
In \cite{CDFV24}, for $\tau=1$, $\xi = 0$, and $\ell=1$ in \eqref{eq:model},  a unique classical solution, global and uniformly bounded in time, is obtained under two scenarios: the subquadratic growth rate
\begin{equation} \label{eq:condACGFV}
	1 \le \alpha < 2, \ \beta > \frac{n}{2} + 2 - \alpha
\end{equation}
and the superquadratic growth rate
\begin{equation} \label{eq:condBCGFV}
	2 \le \alpha < 1 + \frac{2\beta}{n}, \ \beta > \frac{n}{2}.
\end{equation}
A comparison with these results is applicable specifically to Theorem \ref{theo:PP}. We observe that while condition \eqref{eq:TheoPPcond} $a)$  recovers the subquadratic growth regime \eqref{eq:condACGFV} when setting $\ell = 1$ and neglecting chemorepulsion ($\rho$), the regimes described in condition \eqref{eq:TheoPPcond} $b)$ represent a significant departure from the superquadratic case \eqref{eq:condBCGFV}. Notably, in our framework, the proliferation exponent $\alpha$ is no longer required to be upper-bounded by a value dependent on $\beta$. This suggests that sufficiently strong nonlocal damping can compensate for arbitrarily large local growth, as long as the nonlocal death rate $\beta$ is sufficiently large to dominate the system's dynamics.
\end{remark}
}

The remainder of this paper is organized as follows. In Section \ref{sec:Prelim}, we collect  preliminary results, including elliptic and parabolic regularity estimates and fundamental functional inequalities that will be utilized throughout our analysis. Section \ref{sec:localEx} is devoted to the local existence of solutions; here, we establish the positivity and regularity of the local solutions and formulate the extensibility criterion that serves as the foundation for our global analysis. Finally, in Section \ref{sec:estimates}, we derive the necessary a priori estimates to bridge the gap between local and global theory. By proving that local solutions bounded in suitable $L^k(\Omega)$ spaces can be extended to global-in-time classical solutions, we provide the rigorous proofs for our main theorems regarding the global existence and boundedness of solutions.

\section{Preliminaries}
\label{sec:Prelim}
This section lays the mathematical foundation for our analysis, collecting key functional inequalities and regularity estimates used in the following sections. We start with general results for elliptic and parabolic equations (Lemmas \ref{lem:uKw} and \ref{lem:parabReg}) and then present two results from functional interpolation (Lemmas \ref{lem:InterA} and \ref{lem:InterB}), highlighting the role of certain assumptions in Theorems \ref{theo:PE} and \ref{theo:PP} and relevant model parameters.
{
\begin{remark}[Notation]
Throughout this work, $c_i$ (with $i\in\N$) denote generic positive constants whose specific values may change from line to line.  The symbols $\varepsilon_i$ refer to small positive values that can be chosen arbitrarily. 
In particular, with regard to the $\varepsilon_i$, sums between such parameters or multiplications of these by constants are not renumbered. 

We shall denote by $k_0 > 1$ a sufficiently large constant such that all the subsequent arguments requiring $k > k_0$ are valid. In particular, since none of our reasoning imposes an upper bound restriction on $k$, the constant $k_0$ may be tacitly enlarged whenever necessary.
\end{remark}
}

\begin{lemma} \label{lem:uKw}
	{Let $0\leq g\in C^1([0,\infty))$ be such that the lower bound in \eqref{eq:fg} is complied. Additionally let $\varphi,\psi \in C^{2,1}(\bar{\Omega}\times(0,T))$, for some $T>0$, being $\psi$ the} solution of the problem $0 = \Delta \psi - \psi + g(\varphi)$, {{equipped with the condition $\frac{\partial \psi}{\partial\nu} = 0$ on $\partial\Omega\times(0,T)$.}} Moreover, let us require that there exists a constant $M$ satisfying $\| \varphi \|_{L^1(\Omega)} \le M$ for all $t\in(0,T)$. 
	Then, for {for all $\rho> 1$, there exists $k$ sufficiently large, such that for all $\Cl[eps]{e:uKw}>0$ and $\Cl[eps]{e:uKwGN}>0$ it holds}
	\begin{equation} \label{eq:LemuKw}
		\intO \varphi^k \psi \leq \Cr{e:uKw} \int_\Omega \varphi^{k+\rho} + \Cr{e:uKwGN} \int_\Omega |\nabla \varphi^{k/2}|^2 + \Cl[const]{c:uKwGN}, \qquad \text{for all}\ t\in(0,T).
	\end{equation}

	\begin{proof}
    First of all, from \cite[Lemma 2.2]{Winkler14} and \cite[Lemma 3.1]{Viglia21}, one can see that for all $p>1$ and $\varepsilon_0>0$, there exists a positive constant $c := c(p,\varepsilon_0)$ such that
    \begin{equation} \label{ineq:FVg}
		\intO \psi^{p+1} \le \varepsilon_0 \intO g(\varphi)^{p+1} + c \left(\intO g(\varphi)\right)^{p+1}, \quad \forall\ t\in (0,T).
	\end{equation}

    Thereafter, using the Young inequality and \eqref{ineq:FVg} with $p=\frac{k}{\rho}$, we deduce
	\begin{equation} \label{eq:uKw}
		\begin{split}
			\int_\Omega \varphi^k \psi \le& \Cl[eps]{e:uKwt} \int_\Omega \varphi^{k+\rho} + \Cl[const]{c:uKwYoung} \int_\Omega \psi^{\frac{k+\rho}{\rho}} \\
			\le& \Cr{e:uKwt} \int_\Omega \varphi^{k+\rho} + \varepsilon_0 \int_\Omega g(\varphi)^{\frac{k+\rho}{\rho}} + \Cl[const]{c:uKwLem} \left(\int_\Omega g(\varphi)\right)^{\frac{k+\rho}{\rho}},
		\end{split}
	\end{equation}
    for every $t\in(0,T)$. 
    
From the bounds for $g$ in \eqref{eq:fg} and the Jensen inequality $(a+b)^s/2^s \le (a^s + b^s)/2$ for $s>1, a,b>0$, as $\rho > 1$, we obtain on $(0,T)$ that
\begin{align*}
	\int_\Omega g(\varphi)^{\frac{k+\rho}{\rho}} &\le K_2^{\frac{k+\rho}{\rho}}  \int_\Omega (\varphi+1)^{k+\rho} \le 2^{k+\rho-1} K_2^{\frac{k+\rho}{\rho}} \int_\Omega   (\varphi^{k+\rho} + 1), \\
	\int_\Omega g(\varphi) &\le K_2 \int_\Omega (\varphi+1)^\rho \le 2^{\rho-1} K_2 \int_\Omega (\varphi^\rho +1).
\end{align*}
Substituting both in \eqref{eq:uKw} it follows
\begin{equation*}
	\begin{split}
		\intO \varphi^k \psi \le& (\Cr{e:uKwt} + \varepsilon_0) \int_\Omega \varphi^{k+\rho} + \Cr{c:uKwLem} \left( \int_\Omega (\varphi^\rho +1) \right)^{\frac{k+\rho}{\rho}} + \Cl[const]{c:uKwLemII} \\
		\le& \Cr{e:uKw} \int_\Omega \varphi^{k+\rho} + \Cl[const]{e:uKwLemA} \left[ \int_\Omega \varphi^\rho \right]^{\frac{k+\rho}{\rho}} + \Cl[const]{c:uKwJen} \\
		\le& \Cr{e:uKw} \int_\Omega \varphi^{k+\rho} + \Cr{e:uKwGN} \int_\Omega |\nabla \varphi^{k/2}|^2 + \Cr{c:uKwGN},
	\end{split}
\end{equation*}
for all $t\in(0,T)$.

We make precise that in the last step, we have estimated the integral $\left[ \int_\Omega \varphi^\rho \right]^{\frac{k+\rho}{\rho}}$ by means of the Gagliardo-Nirenberg inequality in \cite{LiLankeit2016} and the Young inequality, and relying on the crucial assumption that $\int_\Omega \varphi \leq M$ for all $t\in (0,T)$. More precisely, recalling $\rho > 1$, we have made use of
\begin{equation} \label{eq:GNLER}
	\begin{split}
		\lVert \varphi \rVert_{L^\rho(\Omega)}^{k+\rho} = \lVert \varphi^{\frac{k}{2}} \rVert_{L^\frac{2\rho}{k}(\Omega)}^{\frac{2(k+\rho)}{k}} 
	&\le \left(c_{GN} \|\nabla \varphi^{\frac{k}{2}}\|_{L^2(\Omega)}^{\theta} \|\varphi^{\frac{k}{2}}\|_{L^{\frac2k}(\Omega)}^{(1-\theta)}
	+ c_{GN} \|\varphi^\frac{k}{2}\|_{L^{\frac2k}(\Omega)}\right)^{\frac{2(k+\rho)}{k}} \\&\le 2^\frac{k+2\rho}{k} c_{GN}^\frac{2(k+\rho)}{k} \left(\|\nabla \varphi^{\frac{k}{2}}\|_{L^2(\Omega)}^{\frac{2(k+\rho)\theta}{k}} \|\varphi\|_{L^1(\Omega)}^{(k+\rho)(1-\theta)} + \|\varphi\|_{L^1(\Omega)}^{(k+\rho)}\right) \\
	&\le 2^\frac{k+2\rho}{k} c_{GN}^\frac{2(k+\rho)}{k} \left(\|\nabla \varphi^{\frac{k}{2}}\|_{L^2(\Omega)}^{\frac{2(k+\rho)\theta}{k}} M^{(k+\rho)(1-\theta)} + M^{(k+\rho)}\right) \\
    &\le \Cr{e:uKwGN} \int_\Omega |\nabla \varphi^{k/2}|^2 + \Cl[const]{c:uKwGNa},
	\end{split}
\end{equation}
	being $\theta = \frac{k(1-1/\rho)}{k-1+2/n} \in (0,1)$ and $\frac{(k+\rho)\theta}{k} \in (0,1)$ for all $k>k_0$, and $c_{GN}$ a positive constant. 
\end{proof}
\end{lemma}

\begin{lemma}[Parabolic Regularity] \label{lem:parabReg}
	Let $n \in \N$, $\Omega \subset \R^n$ be a bounded and smooth domain and $q \in (1, \infty)$. Moreover, let $\psi_0 \in W^{2,q}(\Omega)$ such that $\frac{\partial \psi_0}{\partial\nu} = 0$ on $\partial\Omega$. Then, there is a constant $C_{MR} > 0$ such that the following holds: Whenever $T \in (0, \infty]$, $I = [0, T)$, $h \in L^q(I; L^q(\Omega))$, every solution $\psi \in W^1_{loc}(I; L^q(\Omega)) \cap L^q_{loc}(I; W^{2,q}(\Omega))$ of
	\begin{align*}
		&\psi_t = \Delta \psi - \psi + h \quad \text{in } \Omega \times (0, T), \\
		&\psi(\cdot, 0) = \psi_0 \quad \text{in } \Omega,
		\quad \frac{\partial \psi_0}{\partial \nu} = 0 \quad \text{on } \partial\Omega \times (0, T),
	\end{align*}
	satisfies
	\[
	\int_0^t e^s \left( \int_\Omega |\Delta \psi(\cdot, s)|^q \, ds \right) \, ds \le C_{MR} \left[ 1 + \int_0^t e^s \left( \int_\Omega |h(\cdot, s)|^q \, ds \right) \, ds \right] \quad \text{for } 0 < t < T.
	\]
	\begin{proof}
		The proof can be found in \cite[Lemma 3.2]{CDFV24}.
	\end{proof}
\end{lemma}

\begin{lemma} \label{lem:InterA}
	Let $\varphi\in C^{2,1}(\bar{\Omega}\times(0,T))$ be a nonnegative function, for some $T>0$. {{Additionally, for $\alpha,\beta>1$ assume}} 
	\begin{equation*}
		\beta > \frac{n(\alpha-1)}{2}. 
	\end{equation*}
	Moreover, let us require that there exists a constant $M$ satisfying $\| \varphi \|_{L^1(\Omega)} \le M$ for all $t\in(0,T)$. Then {{for all parameters $\Cl[eps]{e:lemInterI}>0$ and $\Cl[eps]{e:lemInterII}>0$ it holds, for all $k>k_0$, that}}
	\begin{equation} \label{eq:InterA}
		\begin{split}
			\int_\Omega \varphi^{k+\alpha-1} \le \Cr{e:lemInterI} \int_\Omega |\nabla \varphi^{k/2}|^2 + \Cr{e:lemInterII} \int_\Omega \varphi^{k+\alpha-1} \int_\Omega \varphi^\beta + \Cl[const]{c:lemInter}, \quad \text{for all} \; t\in(0,T).
		\end{split}
	\end{equation}

	\begin{proof}
		 {
		The estimate follows the interpolation strategy detailed in the proof of \cite[Lemma 5.2]{CDFV24} with the Gagliardo--Nirenberg-type inequality in \cite[Lemma 3.1]{CDFV24}. In this case, thanks to the condition $\beta > \frac{n(\alpha-1)}{2}$, we can obtain the inequality in \cite[(33)]{CDFV24} 
        reading instead
		\begin{equation*}
			\intO \varphi^{k+\alpha-1} \le  \frac{\Cr{e:lemInterI}}{2} \int_\Omega |\nabla \varphi^{k/2}|^2 + \frac{\Cr{e:lemInterI}}{2} \int_\Omega \varphi^{k} + \Cr{e:lemInterII} \int_\Omega \varphi^{k+\alpha-1} \int_\Omega \varphi^\beta + \Cl[const]{c:lemInterP}, \quad \forall\ t\in (0,T),
		\end{equation*}
		for certain positive parameters $\Cr{e:lemInterI}$ and $\Cr{e:lemInterII}$.
	}
		
		 {
		Let us consider as instance of the Gagliardo-Nirenberg inequality in \cite[Lemma 2.3]{LiLankeit2016} 
		\begin{equation} \label{eq:GNiik}
			\begin{split}
			\int_\Omega \varphi^{k} = \lVert \varphi^\frac{k}{2} \rVert_{L^2(\Omega)}^2 &\le c_{GN}^2 \left(\lVert \nabla \varphi^{\frac{k}{2}}\rVert_{L^2(\Omega)}^\theta \lVert \varphi^{\frac{k}{2}}\rVert_{L^\frac{2}{k}(\Omega)}^{1-\theta} + \lVert \varphi^{\frac{k}{2}}\rVert_{L^\frac{2}{k}(\Omega)}  \right)^2 \\
			&\le 2 c_{GN}^2 \lVert \nabla \varphi^{\frac{k}{2}}\rVert_{L^2(\Omega)}^\theta \lVert \varphi \rVert_{L^1(\Omega)}^{k(1-\theta)} + 2 c_{GN}^2 \lVert \varphi \rVert_{L^1(\Omega)}^{k} 		
			\end{split}	
		\end{equation}
		for $\theta = \frac{\frac{k}{2} - \frac12}{\frac{k}{2} - \frac12 + \frac1n} \in (0,1)$ and certain positive constant $c_{GN}$. Reasoning as in \eqref{eq:GNLER}, the hypothesis of boundedness for $\| \varphi \|_{L^1(\Omega)}$ and the Young inequality leads to the thesis.
	}
	\end{proof}
\end{lemma}

Analogous to previous Lemma \ref{lem:InterA}, we can state the following.
\begin{lemma} \label{lem:InterB}
	Let $\varphi\in C^{2,1}(\bar{\Omega}\times(0,T))$ be a nonnegative function, for some $T>0$. {{In addition, for $\alpha,\beta>1$ assume}} 
	\begin{equation*}
		\ell \ge \alpha-1 \qquad \text{and} \qquad \beta > \frac{n\ell + 2(\ell + 1-\alpha)}{2}. 
	\end{equation*}
	Moreover, let us require that there exists a constant $M$ satisfying $\| \varphi \|_{L^1(\Omega)} \le M$ for all $t\in(0,T)$. Then {{for all parameters $\Cl[eps]{e:lemInterIB}>0$ and $\Cl[eps]{e:lemInterIIB}>0$ it holds, for all $k>k_0$, that}}
	\begin{equation} \label{eq:InterB}
		\int_\Omega \varphi^{k+\ell} \le \Cr{e:lemInterIB} \int_\Omega |\nabla \varphi^{k/2}|^2 + \Cr{e:lemInterIIB} \int_\Omega \varphi^{k+\alpha-1} \int_\Omega \varphi^\beta + \Cl[const]{c:lemInterB}, \quad \text{for all} \; t\in(0,T).
	\end{equation}

	\begin{proof}
		 {
			Following the same analytic procedure as in Lemma \ref{lem:InterA}, 
			the derivation can be viewed as a specialization of the functional framework provided in \cite[Lemma 6.2]{DiFraVig} for the inequality in \cite[(44)]{DiFraVig}  by setting the diffusion parameters therein to $m_1 = m_2 = 1$. Utilizing the specific {{Gagliardo-Nirenberg variant provided in}} \cite[Lemma 4.2]{DiFraVig} and the Young inequality, the condition $\beta > \frac{n\ell + 2(\ell + 1 - \alpha)}{2}$ allows to obtain \eqref{eq:InterB} as done for the previous Lemma \ref{lem:InterA}. 
		}
	\end{proof}
\end{lemma}

\section{Local existence and boundedness criterions}
\label{sec:localEx}
In this section, we study the system’s short-time behavior, proving the existence of a unique nonnegative local classical solution and its mass-conservation property. We also present a based-Moser-type criterion, linking local well-posedness to global boundedness.
	\begin{lemma}[Local existence and mass boundedness]\label{lem:local_existence}
		For $\tau\in\{0,1\}$, let all the assumptions in \eqref{eq:modelCond} be fulfilled. 
		Then problem (1) has a unique and nonnegative classical solution
		\begin{gather*}
			%
			u \in C^{2+\sigma,1+\frac\sigma2}(\bar{\Omega}\times[0,\tmax)), \;
			v,w \in C^{2+\sigma,\tau+\frac\sigma2}(\bar{\Omega}\times[0,\tmax)),
		\end{gather*}
		for some maximal $T_{\max} \in (0, \infty]$ which is such that
		\begin{equation} \label{eq:Tmax}
			\text{either } T_{\max} = \infty \quad \text{or } \limsup_{t \to T_{\max}} \|u(\cdot,t)\|_{L^\infty(\Omega)} = \infty. 
		\end{equation}
		Additionally, there exists $M_0 > 0$ such that
		\begin{equation} \label{eq:Massbound}
			\int_\Omega u(x,t) \, dx \le M_0 \quad \text{for all } t \in (0, T_{\max}). 
		\end{equation}
		
		\begin{proof}
		The proof follows the same arguments as in 
		\cite[Proposition 4]{Bian18}
		and \cite[Lemma 7.1]{BaghaeiFraYuyaGiupe2025}.
		\end{proof}
	\end{lemma}

    {Once the local existence of the examined model is guaranteed by the preceding lemma, it is necessary to implement a criterion that ensures the uniform boundedness of these solutions. In this sense, the following result establishes a connection between a precise a priori estimate in a Lebesgue space and the desired boundedness. In this sense, from here on $(u, v, w)$ will indicate the local solution of model \eqref{eq:model} guaranteed in Lemma \ref{lem:local_existence}.}
	
	\begin{lemma}[Boundedness criterion] \label{lem:extcrit}
		If for some $k > k_0$ it holds that
		\[
		u\in L^\infty\left((0,\tmax); L^k(\Omega)\right),
		\]
		actually $u$ is uniformly bounded on $(0, T_{\max})$, and consequently $u \in L^\infty\left((0, \infty); L^\infty(\Omega)\right)$. As a consequence, $v$ and $w$ are also uniformly bounded.
		\begin{proof}
			%
			 {
				The proof follows similar arguments as \cite[Lemma 4.2]{CDFV24} and is based on the regularity-extension framework established in \cite[Lemma A.1]{TaoWinkler2012}. Specifically, the uniform $L^k(\Omega)$ estimates for $u$, for sufficiently large $k$, allow for the application of elliptic or parabolic regularity results to the signaling equations for $v$ and $w$. For $k > n$, standard Sobolev embeddings then imply the uniform boundedness of the signal gradients $\nabla v$ and $\nabla w$ in $L^\infty(\Omega)$. This level of regularity is sufficient to bootstrap the $L^\infty$-norm of $u$ via a Moser-Alikakos-type iteration, thereby preventing the blow-up {{scenario, and ensuring the uniform-in-time boundedness of $u$ on $(0,T_{\max})$. Subsequently, the alternative criterion \eqref{eq:Tmax} provides $T_{\max}=\infty$.}}}
		\end{proof}
	\end{lemma}

\section{A priori estimates. Proof of the Theorems}	
\label{sec:estimates}	
	The core of our global analysis lies in the derivation of uniform-in-time $L^k(\Omega)$ estimates for the cell density. The estimates derived in this section will directly facilitate the application of the extensibility criterion established in Lemma \ref{lem:extcrit}, thereby concluding the proofs of Theorems \ref{theo:PE} and \ref{theo:PP}. 
		
	Our analysis will rely on the behavior of the following  energy functional $\Phi_\tau(t) := e^{\tau t} \int_\Omega u^k$. We divide the study in two cases, whether $\tau = 0$ or $\tau = 1$. 
	 
	 \subsection{Elliptic case ($\tau = 0$)}
     \label{subsec:PE}
	 First, we focus on the parabolic-elliptic model, distinguishing each  parameter regime in an associated lemma.
     
\begin{lemma} \label{lem:ElcasA}
	 	Assume that $\alpha,\beta>1$ and $\ell>0$ satisfying 
	 	\begin{equation} \label{eq:condA}
	 		\ell \le \min\{\alpha - 1,\rho\} \qquad \text{and} \qquad \beta > \frac{n(\alpha-1)}{2}.
	 	\end{equation}
	 	Then, for all $k>k_0$ it holds
	 	\begin{math}
	 		u\in L^\infty\left((0,\tmax); L^k(\Omega)\right).
	 	\end{math}
	 				
	 \begin{proof}	
	 	Let us divide the first condition in \eqref{eq:condA} in two cases for the proof.
	 	
	 	\begin{description}
	 		\item[Case $\ell \le \alpha - 1 = \min\{\alpha - 1,\rho\}$:\ ] \par 
	 Computing the derivative of $\Phi_0(t)$ and using repeated times integration by parts with the homogeneous Neumann boundary conditions on $u,v$, and $w$ we obtain from the first equation in \eqref{eq:model} for all $t\in (0,\tmax)$ that
	\begin{equation} \label{eq:derPhi}
		\begin{split}
			\frac{d \Phi_0(t)}{dt}  &= k \intO u^{k-1} u_t \\
            &= k \intO u^{k-1} \left(\Delta u - \chi \nabla (u\nabla v) + \xi \nabla (u\nabla w) + a u^\alpha - b u^\alpha \intO u^\beta\right)\\ 
		&= -\frac{4(k-1)}{k} \int_\Omega |\nabla u^{k/2}|^2 + \chi k(k-1) \int_\Omega u^{k-1} \nabla u \cdot \nabla v\\
		&\quad - \xi k(k-1) \int_\Omega u^{k-1} \nabla u \cdot \nabla w + a k \int_\Omega u^{k+\alpha-1} - b k \int_\Omega u^{k+\alpha-1} \intO u^\beta \\
		&= -\frac{4(k-1)}{k} \int_\Omega |\nabla u^{k/2}|^2 + \chi (k-1) \int_\Omega \nabla u^k \cdot \nabla v - \xi (k-1) \int_\Omega \nabla u^k \cdot \nabla w \\
		&\quad + a k \int_\Omega u^{k+\alpha-1} - b k \int_\Omega u^{k+\alpha-1} \intO u^\beta \\
		&= -\frac{4(k-1)}{k} \int_\Omega |\nabla u^{k/2}|^2 - \chi (k-1) \int_\Omega u^k \Delta v + \xi (k-1) \int_\Omega u^k \Delta w\\
		&\quad + a k \int_\Omega u^{k+\alpha-1} - b k \int_\Omega u^{k+\alpha-1} \intO u^\beta.
		\end{split}
	\end{equation}
	
	From the elliptic equation for $v$ and $w$ in \eqref{eq:model} we obtain on $(0,\tmax)$
	\begin{equation} \label{eq:ineqderPhi}	
		\begin{split}
		\frac{d \Phi_0(t)}{dt}  &= -\frac{4(k-1)}{k} \int_\Omega |\nabla u^{k/2}|^2 - \chi (k-1) \int_\Omega u^k v + \chi(k-1) \int_\Omega u^k f(u) \\
		&\quad + \xi (k-1) \int_\Omega u^k w - \xi (k-1) \int_\Omega u^k g(u) + a k \int_\Omega u^{k+\alpha-1} \\
        &\quad - b k \int_\Omega u^{k+\alpha-1} \intO u^\beta  \\
		&\le -\frac{4(k-1)}{k} \int_\Omega |\nabla u^{k/2}|^2 + \chi (k-1) K_1 \int_\Omega u^{k+\ell} + \xi(k-1) \int_\Omega u^k w \\
		&\quad - \xi (k-1) K_2 \int_\Omega u^{k+\rho} + a k \int_\Omega u^{k+\alpha-1} - b k \int_\Omega u^{k+\alpha-1} \intO u^\beta,
		\end{split}
	\end{equation}
	after neglecting a nonpositive term and using the bounds in \eqref{eq:fg}.
	
	Let us choose in \eqref{eq:LemuKw} the parameters $\varepsilon_1= K_2$ and $\varepsilon_2 = \frac{2}{k \xi}$. 	
	Substituting in \eqref{eq:ineqderPhi} the resulting inequality we obtain
	\begin{equation} \label{eq:befSimp}
		\begin{split}
			\frac{d \Phi_0(t)}{dt}  \le & -\tfrac{2(k-1)}{k} \int_\Omega |\nabla u^{k/2}|^2 + \chi(k-1) K_1 \int_\Omega u^{k+\ell} + a k \int_\Omega u^{k+\alpha-1} \\
			& - b k \int_\Omega u^{k+\alpha-1} \intO u^\beta + \Cl[const]{c:befSimp}
		\end{split}
	\end{equation}
    for every $t\in(0,\tmax)$.

	Due to the hypothesis $\ell \le \alpha -1$, we can use the Young inequality to obtain
	\begin{equation*}
		\int_\Omega u^{k+\ell} \le  \int_\Omega u^{k+\alpha-1} + \Cl[const]{c:YoungEA}, \quad \forall \ t \in (0,\tmax).
	\end{equation*}
	That inequality merged into \eqref{eq:befSimp} leads to
	\begin{equation}  \label{eq:befInter}
			\frac{d \Phi_0(t)}{dt}  \le -\tfrac{2(k-1)}{k} \int_\Omega |\nabla u^{k/2}|^2 + \Cl[const]{c:casAY} \int_\Omega u^{k+\alpha-1} - b k \int_\Omega u^{k+\alpha-1} \intO u^\beta + \Cl[const]{c:SimpI}
	\end{equation}
    for all $t \in (0,\tmax)$.
	
	Thanks to the hypothesis $\beta > \frac{n(\alpha-1)}{2}$ in \eqref{eq:condA} we can use {  Lemma \ref{lem:InterA}}. With the choice of the parameters $\Cr{e:lemInterI} = \frac{k-1}{ \Cr{c:casAY} k}$ and $\Cr{e:lemInterII}= \frac{bk}{\Cr{c:casAY}}$ in \eqref{eq:InterA}, we can obtain
	\begin{equation} \label{eq:ODIbefGN}
		\frac{d \Phi_0(t)}{dt}  \le - \tfrac{k-1}{k} \int_\Omega |\nabla u^{k/2}|^2 + \Cl[const]{c:ODIA}, \quad \forall \ t \in (0,\tmax).
	\end{equation}
	
	A similar reasoning as in \eqref{eq:GNiik} with the Gagliardo-Nirenberg and the Young inequalities, this time with the bound in \eqref{eq:Massbound}, assures that there exist $\Cl[const]{c:GNuk} >0$ such that it holds on $(0,\tmax)$
	\begin{equation} \label{eq:GNuk}
		\begin{split}
			\intO u^k &\le 2 c_{GN}^2 M_0^{k(1-\theta)} \left(\intO |\nabla u^{k/2}|^2\right)^\theta + 2 c_{GN}^2 M_0^{k} \\
					  &\le \intO |\nabla u^{k/2}|^2 + \Cr{c:GNuk}.
		\end{split}
	\end{equation}
	Inserting that relation in \eqref{eq:ODIbefGN} we obtain the ordinary differential inequality
	\begin{equation*}
		\frac{d \Phi_0(t)}{dt}  \le - \tfrac{k-1}{k} \Phi_0(t) + \Cl[const]{c:ODIA2}, \quad \forall \ t \in (0,\tmax),
	\end{equation*}
	with initial condition $\Phi(0) = \intO u_0^k$. Therefore, there exists a constant $L := \max\left\{\frac{\Cr{c:ODIA2} k}{k-1}, \intO u_0^k\right\}$ such that $\Phi_0(t) = \int_\Omega u^k \le L$ for all $t\in(0,\tmax)$.
	 	
	\item[Case $\ell \le \rho = \min\{\alpha - 1, \rho\}$ \ ] \par  
		If we consider at first $\ell < \rho$, the proof follows the same reasoning as for the first case ($\ell \le \alpha-1$), this time merging into \eqref{eq:ineqderPhi} the Young inequality
		\begin{equation*}
			\int_\Omega u^{k+\ell} \le \frac{\xi K_2}{2\chi K_1} \int_\Omega u^{k+\rho} + \Cl[const]{c:YoungEC}, \quad \forall \ t \in (0,\tmax).
		\end{equation*}
		Therefore, as a consequence of the choice $\Cr{e:uKw} = \frac{K_2}{2}$ and $\Cr{e:uKwGN} = \frac{2}{k \xi}$ in \eqref{eq:LemuKw}, we have on $(0,\tmax)$
		\begin{equation*}
			\frac{d \Phi_0(t)}{dt}  \le -\tfrac{2(k-1)}{k} \int_\Omega |\nabla u^{k/2}|^2 + ak \int_\Omega u^{k+\alpha-1} - b k \int_\Omega u^{k+\alpha-1} \intO u^\beta + \Cl[const]{c:SimpIC}.
		\end{equation*}
		
		From here on, the same procedure used after \eqref{eq:befInter} leads to the claimed result.
		
		On the other hand, if we have $\ell = \rho$, we observe that the condition $\ell \le \alpha-1$ remains valid. Thus, the argument proceeds exactly as in the previous case, beginning from equation \eqref{eq:befSimp}.
	 \end{description}
	
	\end{proof}
\end{lemma}

	A similar result to the previous lemma can be obtained imposing different restrictions on the problem parameters. The following lemma illustrates such different cases.

\begin{lemma} \label{lem:ElcasB}
Assume that $\alpha,\beta,\rho>1$ and $\ell>0$ {\textcolor{blue}{satisfy}} 
\begin{equation} \label{eq:condB}
	\alpha - 1 < \min\{\ell, \rho\} \qquad \text{and} \qquad \beta > \frac{n\ell + 2(\ell-\alpha+1)}{2}.
\end{equation}
Then, for all $k>k_0$ it holds
\begin{math}
	u\in L^\infty\left((0,\tmax); L^k(\Omega)\right).
\end{math}

\begin{proof} Let us consider two cases from the first condition in \eqref{eq:condB} as for the previous lemma, the proof follows the same reasoning as in the previous one. Let us point the differences due to the imposed requirements from \eqref{eq:befSimp} on.
	
\begin{description}
	\item[Case $\alpha - 1 < \ell$: \ ] \par 
	
In this case, from the Young inequality we obtain
\begin{equation*}
	\int_\Omega u^{k+\alpha-1} \le \int_\Omega u^{k+\ell} + \Cl[const]{c:YoungEB}, \quad \forall \ t \in (0,\tmax).
\end{equation*}

Merging that relation into \eqref{eq:befSimp} results
\begin{equation} \label{eq:befInterB}
		\frac{d \Phi_0(t)}{dt}  \le -\tfrac{2(k-1)}{k} \int_\Omega |\nabla u^{k/2}|^2 + \Cl[const]{c:casBY} \int_\Omega u^{k+\ell} - b k \int_\Omega u^{k+\alpha-1} \intO u^\beta + \Cl[const]{c:SimpIB},
\end{equation}
for all $t \in (0,\tmax)$.

This time, the requirement $\beta > \frac{n\ell + 2(\ell-\alpha+1)}{2}$ in \eqref{eq:condB} allows the use of {  Lemma \ref{lem:InterB}}. By choosing the parameters $\Cr{e:lemInterIB}= \frac{k-1}{\Cr{c:casBY}k}$ and $\Cr{e:lemInterIIB}=\frac{bk}{\Cr{c:casBY}}$ in \eqref{eq:InterB}, we have
\begin{equation*} 
	\frac{d \Phi_0(t)}{dt}  \le - \tfrac{2(k-1)}{k} \int_\Omega |\nabla u^{k/2}|^2 + \Cl[const]{c:ODIB}, \quad \ \forall t \in (0,\tmax).
\end{equation*}
Thus, an analogous reasoning as from \eqref{eq:ODIbefGN} on proves the claimed result.

\item[Case $\alpha-1 < \rho$: \ ] \par 
Due to the hypotheses and the Young inequality we obtain 
\begin{equation*}
	\int_\Omega u^{k+\alpha-1} \le \frac{\xi (k-1) K_2}{2ak} \int_\Omega u^{k+\rho} + \Cl[const]{c:YoungED}, \quad \forall \ t\in (0,\tmax).
\end{equation*}

After choosing $\Cr{e:uKw} = \frac{K_2}{2}$ and $\Cr{e:uKwGN} = \frac{2}{k \xi}$  in \eqref{eq:LemuKw}, continuing from \eqref{eq:ineqderPhi} we obtain this time, for every $t \in (0,\tmax)$, that
\begin{equation*}
	\frac{d \Phi_0(t)}{dt}  \le -\frac{2(k-1)}{k} \int_\Omega |\nabla u^{k/2}|^2 + \chi(k-1)K_1 \int_\Omega u^{k+\ell} - b k \int_\Omega u^{k+\alpha-1} \intO u^\beta + \Cl[const]{c:SimpID}.
\end{equation*}

From here on, the same procedure used after \eqref{eq:befInterB} proves the claimed result.
\end{description}

\end{proof}
\end{lemma}

\begin{remark}
	In the last lemma, it could be considered the requirement $\alpha-1 \le \ell$ in \eqref{eq:condB}, obtaining still \eqref{eq:befInterB}. However, in the limit case $\ell = \alpha -1$, the condition in \eqref{eq:condB} regarding the parameter $\beta$ would be the same as in {  Lemmas \ref{lem:InterA} and \ref{lem:ElcasA}}.
\end{remark}

\subsection{Parabolic case ($\tau=1$)}
We now turn our attention to the fully parabolic regime $\tau=1$ in \eqref{eq:model}. In this case,  the energy functional turns into $\Phi_1(t) := e^t \intO u^k$, and the derivative of this functional shares terms with \eqref{eq:derPhi}, as 
\begin{equation} \label{eq:derPsi}
	\frac{d \Phi_1(t)}{dt} = \Phi_1(t) + e^{t} \frac{d \Phi_0(t)}{dt} .
\end{equation}
\begin{lemma}
	For the functional $\Phi_1(t)$, provided $u$ is a solution of \eqref{eq:model} and $\ell>0,\rho>1$, it holds
	\begin{equation} \label{eq:befCasesParb}
		\begin{split}
			\Phi_1(t) \le& \int_0^t e^s \left(-\tfrac{4(k-1)}{k} \int_\Omega |\nabla u^{k/2}|^2 + \Cl[const]{c:Parkl} \int_\Omega u^{k+\ell} + \Cl[const]{c:Parkr} \int_\Omega u^{k+\rho} + \int_\Omega u^k \right. \\
			&\left. + a k \int_\Omega u^{k+\alpha-1} - b k \int_\Omega u^{k+\alpha-1} \int_\Omega u^\beta\right) + \Cl[const]{c:Parexp} e^t + \Cl[const]{c:Parconst}
		\end{split}
	\end{equation}
	for all $t\in(0,\tmax)$.
	
	\begin{proof}
		
Combining the inequalities of Cauchy-Schwarz and Young, provided $\ell>0,\rho>1$, there exist some positive constants $\Cl[const]{c:kvl}, \Cl[const]{c:kvlG},\Cl[const]{c:kwr}$, and $\Cl[const]{c:kwrG}$ such that
\begin{equation} \label{ineq:ukDv}
	- \chi (k-1) \intO u^k \Delta v \le \Cr{c:kvl} \intO u^{k+\ell} + \Cr{c:kvlG} \intO \lvert \Delta v \rvert^{\frac{k+\ell}{\ell}}
\end{equation}
and
\begin{equation} \label{ineq:ukDw}
	\xi (k-1) \intO u^k \Delta w \le \Cr{c:kwr} \intO u^{k+\rho} + \Cr{c:kwrG} \intO \lvert \Delta w \rvert^{\frac{k+\rho}{\rho}},
\end{equation}
for all $t\in(0,\tmax)$. Hence, merging \eqref{ineq:ukDv} and \eqref{ineq:ukDw} into \eqref{eq:derPhi} we obtain
\begin{equation} \label{eq:befParReg}
	\begin{split}
		\frac{d \Phi_0(t)}{dt} \le& -\frac{4(k-1)}{k} \int_\Omega |\nabla u^{k/2}|^2 + \Cr{c:kvl} \int_\Omega u^{k+\ell} + \Cr{c:kvlG} \int_\Omega |\Delta v|^\frac{k+\ell}{\ell} + \Cr{c:kwr} \int_\Omega u^{k+\rho} \\
		& + \Cr{c:kwrG} \int_\Omega |\Delta w|^{\frac{k+\rho}{\rho}} + a k \int_\Omega u^{k+\alpha-1} - b k \int_\Omega u^{k+\alpha-1} \int_\Omega u^\beta .
	\end{split}
\end{equation}

Returning to $\Phi_1(t)$, from \eqref{eq:derPsi} and \eqref{eq:befParReg} we have
\begin{equation*}
	\begin{split}
		\frac{d \Phi_1(t)}{dt} \le& \Phi_1(t) + e^t \left(-\tfrac{4(k-1)}{k} \int_\Omega |\nabla u^{k/2}|^2 + \Cr{c:kvl} \int_\Omega u^{k+\ell} + \Cr{c:kvlG} \int_\Omega |\Delta v|^\frac{k+\ell}{\ell}  \right. \\
		& \left.+ \Cr{c:kwr} \int_\Omega u^{k+\rho} + \Cr{c:kwrG} \int_\Omega |\Delta w|^{\frac{k+\rho}{\rho}} + a k \int_\Omega u^{k+\alpha-1} - b k \int_\Omega u^{k+\alpha-1} \int_\Omega u^\beta  \right).
	\end{split}
\end{equation*}

Integrating for $ s \in [0 ,t]$ we obtain on $(0,\tmax)$ that
\begin{equation*}
	\begin{split}
		\Phi_1(t) - \Phi_1(0) &\le \int_0^t e^s \int_\Omega u^k -\tfrac{4(k-1)}{k} \int_0^t e^s \int_\Omega |\nabla u^{k/2}|^2 + \Cr{c:kvl} \int_0^t e^s \int_\Omega u^{k+\ell} \\
		& + \Cr{c:kvlG} \int_0^t e^s \int_\Omega |\Delta v|^\frac{k+\ell}{\ell} + \Cr{c:kwr} \int_0^t e^s \int_\Omega u^{k+\rho} + \Cr{c:kwrG}\int_0^t e^s \int_\Omega |\Delta w|^{\frac{k+\rho}{\rho}} \\
		& + a k \int_0^t e^s \int_\Omega u^{k+\alpha-1} - b k \int_0^t e^s \left(\int_\Omega u^{k+\alpha-1} \int_\Omega u^\beta\right).
	\end{split}
\end{equation*}
Using Lemma \ref{lem:parabReg} for $\psi = v$ or $w$, respectively with $q = \frac{k+\ell}{\ell}$ or $\frac{k+\rho}{\rho}$, and the conditions in \eqref{eq:fg}, we get
\begin{equation*}
	\begin{split}
		\Phi_1(t) \le& \int_\Omega u_0^k + \int_0^t e^s \int_\Omega u^k -\tfrac{4(k-1)}{k} \int_0^t e^s \int_\Omega |\nabla u^{k/2}|^2 + \Cr{c:kvl} \int_0^t e^s \int_\Omega u^{k+\ell} \\
		&  + \Cr{c:kvlG} \int_0^t e^s \int_\Omega f(u)^\frac{k+\ell}{\ell} + \Cr{c:kwr} \int_0^t e^s \int_\Omega u^{k+\rho} + \Cr{c:kwrG}\int_0^t e^s \int_\Omega g(u)^{\frac{k+\rho}{\rho}}\\
		& + a k \int_0^t e^s \int_\Omega u^{k+\alpha-1} - b k \int_0^t e^s \left(\int_\Omega u^{k+\alpha-1} \int_\Omega u^\beta\right) + \Cl[const]{c:kvwlr} \\
		\le& \int_\Omega u_0^k + \int_0^t e^s \int_\Omega u^k -\tfrac{4(k-1)}{k} \int_0^t e^s \int_\Omega |\nabla u^{k/2}|^2 + \Cr{c:kvl} \int_0^t e^s \int_\Omega u^{k+\ell} \\
		& + \Cl[const]{c:kvlGa} \int_0^t e^s \int_\Omega u^{k+\ell} + \Cr{c:kwr} \int_0^t e^s \int_\Omega u^{k+\rho} + \Cl[const]{c:kwrGa} \int_0^t e^s \int_\Omega u^{k+\rho} + \Cl[const]{c:kwrGb} e^t\\
		& + a k \int_0^t e^s \int_\Omega u^{k+\alpha-1}  - b k \int_0^t e^s \left(\int_\Omega u^{k+\alpha-1} \int_\Omega u^\beta\right) + \Cl[const]{c:parreg} \\
		\le& \int_0^t e^s \left( \int_\Omega u^k -\tfrac{4(k-1)}{k} \int_\Omega |\nabla u^{k/2}|^2 + \Cr{c:Parkl} \int_\Omega u^{k+\ell} + \Cr{c:Parkr} \int_\Omega u^{k+\rho} + a k \int_\Omega u^{k+\alpha-1} \right. \\
		&\left. - b k \int_\Omega u^{k+\alpha-1} \int_\Omega u^\beta\right) + \Cr{c:Parexp} e^t + \Cr{c:Parconst},
	\end{split}
\end{equation*}
for every $t\in(0,\tmax)$. Therefore, the lemma holds.
\end{proof}
\end{lemma}

{Having reduced the energy functional to a form compatible with \eqref{eq:befSimp}, we may now proceed by studying different cases as in Section \ref{subsec:PE}. }

\begin{lemma} \label{lem:ParcasA}
	Assume that $\alpha,\beta,\rho>1$ and $\ell>0$ {{satisfy}} 
	\begin{equation} \label{eq:condAParb}
		\max\ \{\rho , \ell\} \le \alpha - 1 \qquad \text{and} \qquad \beta > \frac{n(\alpha-1)}{2}.
	\end{equation}
	Then, for all $k>k_0$ it holds
	\begin{math}
		u\in L^\infty\left((0,\tmax); L^k(\Omega)\right).
	\end{math}
	
	\begin{proof}
		From condition $\max \{\rho,\ell\} \le \alpha-1$ in \eqref{eq:condAParb} and the Young inequality, the inequality \eqref{eq:befCasesParb} leads to
		\begin{equation*}
			\begin{split}
				\Phi_1(t) \le& \int_0^t e^s \left( \int_\Omega u^k -\tfrac{4(k-1)}{k} \int_\Omega |\nabla u^{k/2}|^2 + \Cl[const]{c:PYak} \int_\Omega u^{k+\alpha-1} - b k \int_\Omega u^{k+\alpha-1} \int_\Omega u^\beta\right) \\
				& + \Cl[const]{c:Parexpi} e^t + \Cr{c:Parconst},
			\end{split}
		\end{equation*} 
        for all $t\in(0,\tmax)$.
        
		The condition $\beta > \frac{n(\alpha-1)}{2}$ in \eqref{eq:condAParb} allows to apply {  Lemma \ref{lem:InterA}}. With the choice $\Cr{e:lemInterI}=\frac{2(k-1)}{\Cr{c:PYak}k}$ and $\Cr{e:lemInterIB}=\frac{bk}{\Cr{c:PYak}}$ in \eqref{eq:InterA}, it follows on $(0,\tmax)$ that
		\begin{equation} \label{eq:ParbafterLemIA}
			\begin{split}
				\Phi_1(t) \le& \int_0^t e^s \left(\intO u^k - \tfrac{2(k-1)}{k} \int_\Omega |\nabla u^{k/2}|^2 \right) + \Cl[const]{c:Parexpii} e^t + \Cr{c:Parconst}.
			\end{split}
		\end{equation}
		Reasoning as in \eqref{eq:GNuk}, we know that there exists a positive constant $\Cl[const]{c:GNpar}$ such that
		\begin{displaymath}
				\intO u^k \le \tfrac{2(k-1)}{k} \intO |\nabla u^{k/2}|^2 + \Cr{c:GNpar}.
		\end{displaymath}
		Consequently, we conclude that 
		\begin{math}
			\Phi_1(t) \le \Cl[const]{c:Parexpiii} e^t + \Cl[const]{c:Parconsti}
		\end{math}
		 and this implies that 
         \begin{equation*}
            \intO u^k \le \Cr{c:Parexpiii}  + \Cr{c:Parconsti} e^{-t} \le 2 \max \{\Cr{c:Parexpiii}  , \Cr{c:Parconsti}\} \quad \forall \; t\in(0,\tmax),       
         \end{equation*}
	   so concluding. 
	\end{proof}
\end{lemma}
		
\begin{lemma} \label{lem:ParcasB}
	Assume that $\alpha,\beta,\rho>1$ and $\ell>0$ {{satisfy}}
	\begin{equation} \label{eq:condBParb}
		\alpha - 1 \le \min\ \{\rho , \ell\} \qquad \text{and} \qquad \beta > \frac{n \max\left\{\rho,\ell \right\} + 2(\max\left\{\rho,\ell \right\} - (\alpha-1))}{2}.
	\end{equation}
	Then, for all $k>k_0$ it holds
	\begin{math}
		u\in L^\infty\left((0,\tmax); L^k(\Omega)\right).
	\end{math}
	
	\begin{proof}
		From condition $\alpha-1 \le \min \{\rho,\ell\}$ in \eqref{eq:condBParb} and the Young inequality we have on $(0,\tmax)$ that
		\begin{equation*}
			\intO u^{k+\alpha-1} \le \intO u^{k+\min \left\{\rho,\ell\right\}} + \Cl[const]{c:YoungcaseIIm} \le \intO u^{k+\max \left\{\rho,\ell\right\}} + \Cl[const]{c:YoungcaseIIM}.
		\end{equation*}
		Plugging these results into \eqref{eq:befCasesParb} we obtain due to the Young inequality that
		\begin{equation*}
			\begin{split}
				\Phi_1(t) \le& \int_0^t e^s \left( \int_\Omega u^k -\tfrac{4(k-1)}{k} \int_\Omega |\nabla u^{k/2}|^2 + \Cl[const]{c:PYakB} \int_\Omega u^{k+\max\{\rho,\ell\}} - b k \int_\Omega u^{k+\alpha-1} \int_\Omega u^\beta\right) \\
				& + \Cl[const]{c:Parexpia} e^t + \Cr{c:Parconst},
			\end{split}
		\end{equation*} 
        for all $t\in(0,\tmax)$.
        
		Thanks to the condition
        \begin{math}
            \beta > \frac{n \max\left\{\rho,\ell \right\} + 2(\max\left\{\rho,\ell \right\} - (\alpha-1))}{2}
        \end{math}
         in \eqref{eq:condBParb}, we apply {  Lemma \ref{lem:InterB}} analogously to the previous cases, with $\rho$ replacing $\ell$ in \eqref{eq:InterB} where appropriate. The conclusion of the proof then follows the steps described from \eqref{eq:ParbafterLemIA} onwards.		
	\end{proof}
\end{lemma}		

\subsection{Proof of the main results}

Now we can present the proofs of Theorems \ref{theo:PE} and \ref{theo:PP}.

\begin{proof}[Proof of  Theorem \ref{theo:PE}]
	 {By virtue of the conditions on the exponents $\alpha, \beta, \ell,$ and $\rho$ in cases $a)$ and $b)$ of Theorem \ref{theo:PE} (see \eqref{eq:TheoPEcond}), the main bounds achieved in Lemmas \ref{lem:ElcasA}, \ref{lem:ElcasB}, and Lemma \ref{lem:extcrit} are sufficient to conclude.}
\end{proof}

\begin{proof}[Proof of Theorem \ref{theo:PP}]
	 {The global {{boundedness}} for the fully parabolic regime follows the same logical framework to that of Theorem \ref{theo:PE}. Under the assumptions $a)$ or $b)$ of Theorem \ref{theo:PP} (see \eqref{eq:TheoPPcond}), in this scenario we conclude by invoking the uniform-in-time $L^k(\Omega)$-estimates for the cell density $u$ obtained in Lemmas \ref{lem:ParcasA} and \ref{lem:ParcasB}.}
\end{proof}

\section{Conclusions}

In this work, we have investigated the global dynamics of a three-component attraction-repulsion chemotaxis system featuring a nonlocal logistic-type source term. By performing an analysis of the interplay between nonlinear signaling production and population kinetics, we established sufficient conditions on the exponents $\ell, \rho, \alpha,$ and $\beta$ to ensure the global existence and uniform-in-time boundedness of classical solutions. 

Our results, covering both the parabolic-elliptic ($\tau=0$) and fully parabolic ($\tau=1$) regimes, demonstrate that the aggregative tendencies inherent in chemotactic models can be effectively suppressed by the stabilizing influence of chemorepulsion and nonlocal damping. Specifically, we have shown that a sufficiently strong nonlocal competitive pressure can prevent the formation of finite-time singularities, even in the presence of aggressive signal production. By extending the framework of previous studies to a broader class of nonlinear interactions, this paper provides a more comprehensive characterization of the mechanisms that maintain population stability in complex chemical environments. Future research may further explore the asymptotic behavior of these solutions and the specific thresholds for pattern formation within the identified stability regimes.
\medskip 

\section*{Acknowledgments}
{The authors RDF and GV are members of the Gruppo Nazionale per l'Analisi Matematica, la Probabilit\`a e le loro Applicazioni (GNAMPA) of the Istituto Nazionale di Alta Matematica (INdAM), and participate in the INdAM - GNAMPA Project {\em Modelli di reazione-diffusione-trasporto: dall'analisi alle applicazioni} (CUP E53C25002010001). Both are also partially supported by the research project {\em Partial Differential Equations and their role in understanding natural phenomena} (2023, CUP F23C25000080007), funded by  \href{https://www.fondazionedisardegna.it/}{Fondazione di Sardegna}. 
RDF acknowledges financial support by PNRR e.INS Ecosystem of Innovation
for Next Generation Sardinia (CUP F53C22000430001, codice MUR ECS0000038).
The author MVRN is a member of the research group FQM-315 of Junta de Andalucía and has been partially supported by grants PR2024-011 and PR2024-039  of the ‘Plan Propio–UCA 2025-2026’ funded by Universidad de Cádiz, Spain. 
}

\bibliographystyle{abbrv}
\bibliography{Chemotaxis}
\end{document}